\documentclass{daj}

\usepackage{amsmath}
\usepackage{amsfonts}
\usepackage{amsthm}

\theoremstyle{plain}
\newtheorem{theorem}{Theorem}[section]
\newtheorem{lemma}[theorem]{Lemma}
\newtheorem{corollary}[theorem]{Corollary}
\newtheorem{claim}[theorem]{Claim}
\newtheorem{proposition}[theorem]{Proposition}

\theoremstyle{definition}

\dajAUTHORdetails{%
  title = {On automorphism groups of Toeplitz subshifts}, 
  author = {Sebasti\'an Donoso, Fabien Durand, Alejandro Maass, and Samuel Petite},
  plaintextauthor = {Sebastian Donoso, Fabien Durand, Alejandro Maass, Samuel Petite},
  plaintexttitle = {On automorphism groups of Toeplitz subshifts}, 
  keywords = {Toeplitz subshifts, automorphism group, complexity function, coalescence.},
}   

\dajEDITORdetails{%
   year={2017},
   number={11},
   received={05 January 2017},   
   published={15 June 2017},  
   doi={10.19086/da.1832},       
}   

\begin{document}

\begin{frontmatter}[classification=text]

\title{On automorphism groups of Toeplitz subshifts \thanks{This research was supported through the cooperation project MathAmSud DCS 38889 TM.}} 


\author[Sebastian]{Sebastian Donoso\thanks{ Partially supported by ERC grant 306494 and CMM-Basal grant PFB-03.}}
\author[Fabien]{Fabien Durand}
\author[Alejandro]{Alejandro Maass\thanks{Supported by CMM-Basal grant PFB-03.}}
\author[Samuel]{Samuel Petite}

\begin{abstract}
In this article we study automorphisms of Toeplitz subshifts. 
Such groups are abelian and any finitely generated torsion subgroup is finite and cyclic. 
When the complexity is non-superlinear, we prove that the automorphism group is, modulo a finite cyclic group,  generated by a unique root of the shift.
In the subquadratic complexity case, we show that the automorphism group modulo the torsion is generated by the roots of the shift map and that the result of the non-superlinear case is optimal. Namely,  for any $\varepsilon > 0$ we construct examples of minimal Toeplitz subshifts with complexity bounded by $C n^{1+\epsilon}$ whose automorphism groups are not finitely generated. Finally, we observe the coalescence and the automorphism group give no restriction on the complexity since we provide a family of coalescent Toeplitz subshifts with positive entropy such that their automorphism groups are arbitrary finitely generated infinite abelian groups with cyclic torsion subgroup (eventually restricted to powers of the shift). 		
\end{abstract}
\end{frontmatter}


\section{Introduction}
Given a finite set $\mathcal{A}$, the {\em full shift } ${\mathcal A}^{{\mathbb Z}}$ is the set of all bi-infinite sequences $(x_i)_{i\in \mathbb{Z}}$ with $x_i  \in \mathcal{A}$ for all $i\in \mathbb{Z}$. This set can be thought of as all {\em colorings} of $\mathbb{Z}$ with colors in $\mathcal{A}$.  The finite set $\mathcal{A}$ is usually called the {\em alphabet}.    A \emph{subshift} is a subset of a fullshift ${\mathcal A}^{{\mathbb Z}}$ which is closed for the product topology and invariant under the {\em shift map} $\sigma\colon \mathcal{A}^{\mathbb{Z}}\to \mathcal{A}^{\mathbb{Z}}$, $(x_i)_{i\in \mathbb{Z}}\mapsto (x_{i+1})_{i\in \mathbb{Z}}$.  A {\em word}  $w$ is an element of  $\cup_{n\in \mathbb{N}} \mathcal{A}^n$. Its length $|w|$ is the positive integer such that $w\in \mathcal{A}^{|w|}$.   We say that a word $w$  {\em appears } in a sequence $x=(x_i)_{i \in \mathbb{Z}}\in X$ if there exists $n\in \mathbb{Z}$ such that $w=x_n\cdots x_{n+|w|-1}$. The \emph{complexity} of a subshift $X$ is the map $p_X(\cdot):\mathbb{N}\to \mathbb{N}$ which for $n\in \mathbb{N}$ counts the number of  words of length $n$ appearing in sequences of $X$. An {\em endomorphism} is a continuous onto  map $\phi \colon X \to X$ such that $\phi\circ \sigma=\sigma\circ \phi$. It is called an {\em automorphism} whenever it is bijective. We remark that since $X$ is compact, the inverse of an automorphism $\phi$ is also continuous and thus is an automorphism too.   The group of all automorphisms is countable and denoted ${\rm Aut}(X,\sigma)$.

The study of the group of automorphisms of low complexity subshifts has become very active in the last five years. In contrast with the positive entropy studies, where the automorphism group can be very large (see for instance \cite{BoyleLindRudolph}), a lot of evidence suggests that low complexity systems ought to have a small automorphism group. In particular, studies \cite{SaloTorma} and \cite{Olli} on classes of minimal substitutive subshifts showed that the automorphism groups are virtually $\mathbb{Z}$. It turns out that this result holds for any minimal subshift with non-superlinear complexity (\emph{i.e.}, $\displaystyle\liminf_{n \to \infty} p_X(n)/n<\infty$)  \cite{CyrKra2,DonosoDurandMaassPetite}. Also, for some special substitution subshifts with this complexity growth, the automorphism group is cyclic \cite{CovenQuasYassawi}. 
Higher order polynomial complexity growth was also considered by Cyr and Kra in \cite{CyrKra,CyrKra3}. 
In \cite{CyrKra}, the authors proved that for transitive subshifts, if 
$\displaystyle\liminf_{n \to \infty} p_X(n)/n^2=0$ then the quotient ${\rm Aut}(X,\sigma)/\langle \sigma \rangle$ is a periodic group, where $\langle \sigma \rangle$ is the group spanned by the shift map. In \cite{CyrKra3}, for a large class of minimal subshifts of subexponential complexity they also proved that the automorphism group is amenable.  
These results showed that the automorphism group seems to gain in constraints when the complexity goes down but this is not always true. Interestingly, in \cite{Salo} the author provided a Toeplitz subshift with complexity $p_X(n)\leq Cn^{1.757}$, whose automorphism group is not finitely generated. 

Even though these results allow us to slowly understand the group of automorphisms of low complexity subshifts, the complete picture is still unclear, even for particular classes of subshifts. The purpose of this article is to use the class of Toeplitz subshifts and study their automorphism groups in order to understand some general questions relating complexity growth and the size of the automorphism group. We observe that since the automorphism group of a Toeplitz subshift is a subgroup of the associated odometer \cite{DonosoDurandMaassPetite}, then it is a countable abelian group. This fact restricts our study to such class of groups. 

We start by considering Toeplitz subshifts of subquadratic complexity. In this case, any endomorphism is bijective, so the coalescent property holds (see Section \ref{Sec:ToeplitzSubquadratic} for an expanded discussion). In Theorem \ref{Thm:RootsToeplitz} we show that the automorphism group is spanned by the roots of the shift map modulo the torsion subgroup. In fact, this result holds each time the quotient ${\rm Aut}(X,\sigma)/\langle \sigma \rangle$ is periodic and the subquadratic case is a consequence of \cite{CyrKra}. 
When this quotient is finite, in particular when the Toeplitz subshift has non-superlinear complexity \cite{DonosoDurandMaassPetite}, we also prove in Theorem \ref{Thm:RootsToeplitz} that the automorphism group modulo a finite cyclic group is spanned by one root of the shift map, thus this quotient is a cyclic group. 
Both results open the door to applications to other Toeplitz subshifts with higher complexities. 
These results follow from the study of subgroups of odometers and allow us to recover Theorem 12 from \cite{CovenQuasYassawi}, where the authors consider a family of Toeplitz subshifts generated by substitutions. The condition on Theorem \ref{Thm:RootsToeplitz} is true for many Toeplitz in the family called $(p,q)$-Toeplitz \cite{CK}, including the 
example in \cite{Salo}. Then, using this class of Toeplitz subshifts we extend the example given by Salo in \cite{Salo}, proving that the condition of non-superlinear complexity of  \cite{CyrKra2,DonosoDurandMaassPetite} cannot be relaxed. More precisely, 
in Theorem \ref{Thm:LowComplexityToeplitz} we prove that for any $\varepsilon>0$ there exists a Toeplitz subshift in the aforementioned family such that the complexity verifies $p_X(n)\leq C n^{1+\varepsilon}$ for all $n \in \mathbb{N}$ and whose automorphism group is not  finitely generated. In fact, this construction allows us to produce Toeplitz subshifts with arbirarily big polynomial complexity whose automorphism group is not  finitely generated. We left open the question whether we can get the same result with an even smaller complexity. For instance, $p_X(n)/n\leq  C \log (n)$ for all $n\in \mathbb{N}$. 

In contrast with previous results, large complexity is not enough to have a large automorphism group. Here we prove that coalescence and the size of the automorphism group impose no restrictions to the complexity function. In Theorem \ref{theo:anycountable} we provide a family of coalescent Toeplitz subshifts with positive entropy such that its automorphism group is an arbitrary   infinite countable and finitely generated abelian group with cyclic torsion subgroup. For the group $\mathbb{Z}$, this result resembles the main result in \cite{BulatekK} where the construction is not explicit. Also, in \cite{DownarowiczKwiatkowskiLacroix:1995}, an explicit example is  given, with an arbitrary entropy but with a specific maximal equicontinuous factor. Our construction is explicit, self-contained and may be generalized to $\mathbb{Z}^d$-Toeplitz arrays or $G$-Toeplitz arrays for any countable amenable residually finite group $G$ (see \cite{CortezPetite:2008}). 

Finally, we remark that not every infinite countable abelian group can be the automorphism group of a Toeplitz subshift: it has to be a subgroup of an odometer \cite{DonosoDurandMaassPetite}, hence residually finite. For instance, the rational  numbers $\mathbb{Q}$ does not satisfy this property (see Section \ref{Sec:BackgroundToeplitz}). But besides this restriction, we do not know if every infinitely generated  countable abelian group (embedded in an odometer) can be realized as the automorphism group of a Toeplitz subshift. 

\section{Background}

\subsection{Topological dynamical systems} 

A {\it topological dynamical system} (or just a system) is a pair $(X,T)$ where $X$ is a compact metric space and $T\colon X \to X$ is a homeomorphism. 
Let $\text{``dist''}$ be a distance in $X$. The orbit of a point $x\in X$ is given by $\text{Orb}_T(x)= \{T^nx ; n\in \mathbb{Z}\}$. A topological dynamical system is {\it minimal} if the orbit of every point is dense in $X$ and is {\it transitive} if at least one orbit is dense in $X$. In a transitive system, points with dense orbits are called {\it transitive points}. 

Let $(X,T)$ be a topological dynamical system. We say that $x,y\in X$ are {\em proximal} if there exists a sequence $(n_i)_{i\in \mathbb{N}}$ in $\mathbb{Z}$ such that $\lim_{i\to \infty} \text{dist}(T^{n_i} x, T^{n_i} y)=0.$ A stronger condition than proximality is asymptoticity. Two points $x, y \in X$ are {\em asymptotic} if $\lim_{n\to \infty} \text{dist}(T^n x, T^n y) =0.$
Nontrivial asymptotic pairs may not exist in an arbitrary topological dynamical system but it is well known that a nonempty aperiodic subshift always admits at least one \cite[Chapter 1]{Aus}. 

A {\it factor map} between the topological dynamical systems $(X,T)$ and $(Y,S)$ is a continuous onto map $\pi\colon X\to Y$ such that $\pi\circ T=S\circ \pi$. We say that $(Y,S)$ is a {\it factor} of $(X,T)$ and that $(X,T)$ is an {\it extension} of $(Y,S)$. We use the notation $\pi\colon (X,T)\to (Y,S)$ to indicate the factor map. If in addition $\pi$ is a bijective map we say that $(X,T)$ and $(Y,S)$ are {\it topologically conjugate}. 

The system $(X,T)$ is a {\em proximal extension} of $(Y,S)$ via the factor map \break $\pi\colon (X,T)\to (Y,S)$ (or the factor map itself is a {\it proximal extension}) if for every $x,x'\in X$ the condition $\pi(x)=\pi(x')$ implies that $x,x'$ are proximal.
For minimal systems, $(X,T)$ is an {\em almost one-to-one extension} of $(Y,S)$ via the factor map $\pi:(X,T)\to (Y,S)$ (or the factor map itself is an {\it almost one-to-one extension}) if there exists $y \in Y$ with a unique preimage for the map $\pi$.
The relation between these two notions is given by the following folklore result.
If the factor map $\pi:(X,T)\to (Y,S)$ between minimal systems  
is an almost one-to-one extension then it is also a proximal extension (see \cite{DonosoDurandMaassPetite} for a proof).  

A topological dynamical system $(X,T)$ is \emph{equicontinuous} if for any $\epsilon>0$ there is $\delta >0$ such that  if $\textrm{dist}(x,y)\leq \delta$ then for any $n\in \mathbb{Z}$ one has $\text{dist}(T^nx,T^ny)\leq \epsilon$. It is well known that any topological system has a \emph{maximal equicontinuous factor}, that is, a factor that is equicontinuous and that is an extension of any other equicontinuous factor of the system (see \cite{Aus}).

An {\it automorphism} of the topological dynamical system $(X,T)$ is a homeomorphism $\phi$ of the space $X$ to itself such that $\phi\circ T=T\circ \phi$.
We denote by Aut$(X,T)$ the group of automorphisms of $(X,T)$. The subgroup of Aut$(X,T)$ generated by $T$ is denoted by $\langle T \rangle$. Analogously one defines an {\it endomorphism} of the topological dynamical system $(X,T)$ as a continuous and onto map $\phi\colon X \to X$ such that $\phi\circ T=T\circ \phi$. The space of endomorphisms of $(X,T)$ is denoted End$(X,T)$. A system $(X,T)$ is {\it coalescent} if 
End$(X,T)$=Aut$(X,T)$. When $(X,\sigma)$ is a subshift, the Curtis-Hedlund-Lyndon theorem asserts that an endomorphism $\phi$ is defined by a {\it local rule}. That is, there exists ${\textbf{ r}}\in \mathbb{N}$ (called a radius of $\phi$) and a {\em block map} $\widehat{\phi}\colon \mathcal{A}^{2{\bf r}+1}\to \mathcal{A}$ such that $\phi(x)_{n}=\widehat{\phi}(x_{n-{\bf r}}\ldots x_{n} \ldots x_{n+{\bf r}})$ for every $n\in \mathbb{Z}$. 

\subsection{Odometers}
Let $(p_n)_{n\geq 1}$ be a sequence of natural numbers such that $p_n$ divides $p_{n+1}$ for all 
$n\geq 1$. Define the quotients associated to this sequence by $q_1=p_1$ and $q_{n+1}=p_{n+1}/p_{n}$ for $n\geq 1$. 
The \emph{odometer} at scale $(p_n)_{n\geq 1}$ is given by
\[  \mathbb{Z}_{(p_{n})} = \{ (x_{n})_{n\ge 1 } \in \prod_{n=1}^\infty \mathbb{Z}_{ p_{n}} ; x_{n+1}=  x_n \mod p_{n} \  \forall n \geq 1\},\]
where $\mathbb{Z}_{p}$ stands for $\mathbb{Z} / p\mathbb{Z}$, notation which is usually devoted to the ring of $p$-adic integers when $p$ is prime.
In this framework this ring is the one associated to the scale $(p^n)_{n\geq 1}$ so will be denoted $\mathbb{Z}_{(p^n)}$. 

The set $\mathbb{Z}_{(p_n)}$ is an inverse limit $\lim\limits_{\leftarrow} \mathbb{Z}_{p_n}$ of the canonical homomorphisms $\mathbb{Z}_{p_{n+1}}\to \mathbb{Z}_{p_n}$.
Clearly $\mathbb{Z}_{(p_n)}$ is an abelian group with the coordinatewise addition and when finite it is cyclic. The odometer $\mathbb{Z}_{(p_n)}$ where $p_{n}=n!$ for all $n\geq 1$ is called the {\em universal odometer}. We notice that the odometers $\mathbb{Z}_{(p_n)}$ and $\mathbb{Z}_{(p_{i_n})}$, where 
$(i_n)_{n\geq 1}$ is strictly increasing, are isomorphic as groups.  
We denote by ${\bf 0}$ and ${\bf 1}$ the elements $(0,0,\ldots)$ and $(1,1,\ldots)$ in any odometer. The natural dynamics on an odometer $\mathbb{Z}_{(p_n)}$ is given by the addition by ${\bf 1}$. It is not difficult to see that it is a minimal equicontinuous topological dynamical system, that we also call odometer and denote by $(\mathbb{Z}_{(p_n)}, +{\bf 1})$. In particular, the subgroup $\langle {\bf 1} \rangle$ generated by ${\bf 1}$ is dense in $\mathbb{Z}_{(p_n)}$. This subgroup is identified with the integers. 

The following simple lemma is a slight generalization of the minimality of the odometer. 
The proof is given for completeness.  

\begin{lemma}\label{lem:generalizedminimality}
	Let $\mathbb{Z}_{(p_n)}$ be an odometer and consider an integer $m\in \mathbb{Z}$ such that 
	$(m,p_n)=1$ for all $n\geq 1$. Then, the dynamics defined by the addition by 
	$m{\bf 1}$ in $\mathbb{Z}_{(p_n)}$ is minimal.
\end{lemma}
\begin{proof}
	Since the addition by $\bf 1$ is minimal in $\mathbb{Z}_{(p_n)}$, it suffices to show that the orbit of $0$ in $\mathbb{Z}_{p_n}$ by the addition by $m$ contains $ 1$ for all $n\geq 1$. 
	Since $(m,p_n)=1$, there exist integers $a,  b$ such that 
	$a m=b  p_n +1$.
	This ends the proof.
\end{proof}

We will also need to understand when an odometer has torsion elements. For that we need some extra notation. For each prime number $p$, denote by $v_{p}(n)$ the $p$-{\em adic valuation} of the integer $n$, that is, $v_{p}(n) = \max \{k\geq 0 ; p^{k} \textrm{ divides } n\}$. 
Given an odometer $\mathbb{Z}_{(p_n)}$, the sequence $(v_{p}(p_{n}))_{n\geq 1}$ is a non-decreasing sequence and we can define the {\em multiplicity function}, as proposed in \cite{Downarowicz:2005}, by 
$$ {\bf v}((p_{n})) = \left ( \lim_{n\to \infty} v_{p}(p_{n}); p \textrm{ prime}  \right ).$$

For an abelian group $G$, let $T(G)$ denote its torsion subgroup, that is, the subgroup generated by the torsion elements, {\em i.e.}, elements of finite order. 
For an integer $p$, let $T(G)_{p}$ denote the set of elements in $G$ of order a power of $p$. 

\begin{lemma}\label{lem:torsion}
	Let  $\mathbb{Z}_{(p_n)}$  be an odometer. Then its torsion subgroup is 
	$$T({\mathbb{Z}_{(p_n)}}) = \bigoplus_{p} T({\mathbb{Z}_{(p_n)}})_{p},$$
	where the sum is taken over all the prime numbers $p$ such that $\lim_{n\to \infty} v_{p}(p_{n}) $ is positive and finite. Moreover, each group $T({\mathbb{Z}_{(p_n)}})_{p}$ is a finite cyclic group of order  $p^{\lim_{n\to \infty} v_{p}(p_{n})}$.
	In particular, if $\lim_{n \to \infty} v_{p}(p_{n}) \in \{0,  \infty\}$ for all prime numbers, then $\mathbb{Z}_{(p_n)}$ is torsion free.  
\end{lemma}
It follows from the Chinese remainder theorem  that the torsion subgroup of a finitely generated subgroup of an odometer is cyclic. Also, the group of $p$-adic integers $\mathbb{Z}_{(p^n)}$ is torsion free. In contrast, the odometer $\mathbb{Z}_{(p_n)}$ where $p_{n}$ is the product of the first $n$ primes, {\em i.e.}, $\lim_{n\to \infty} v_{p}(p_{n})=1$ for any prime number $p$, has a non-finitely generated torsion subgroup. 
\begin{proof}
	The Chinese remainder theorem implies that $T({\mathbb{Z}_{(p_n)}}) = \bigoplus_{p} T({\mathbb{Z}_{(p_n)}})_{p},$ where the sum is taken over all the prime numbers. So it suffices to study each group $T({\mathbb{Z}_{(p_n)}})_{p}$.
	
	Let $p$ be  a prime such that there exists $(x_{n})_{n\geq 1} \in  T({\mathbb{Z}_{(p_n)}})_{p}$ different from ${\bf 0}$, of order $p^k$ for some $k \ge1 $, meaning $p^k x_{n} = 0 \mod p_{n}$ for each $n \ge 1$. 
	Moreover there exists $n$ such that for all large enough $m\geq n$, $p^k$ is the order of $x_m$ in $\mathbb{Z}_{p_m}$, and, 
	thus $p^k$ divides $p_{m}$.
	Thus, if $T({\mathbb{Z}_{(p_n)}})_{p}$ is non trivial then  $\lim_{n\to \infty} v_{p}(p_{n})$ is positive.
	Let us show it is finite. 
	If it is not the case then there would be some $m$ such that $p_m = p_npq$ for some integer $q$ and we should have $p^{k-1}x_m = ap_nq$ for some integer $a$.
	Consequently, $p^{k-1}x_n = 0 \mod p_{n}$ and thus $x_n = 0  \mod p_{n}$.
	This would contradict the fact that $p^k$ is the order of $x_n$.
	
	Hence, $\lim_{n\to \infty} v_{p}(p_{n}) =k_{p}$ is finite and the order of $(x_n)_{n\geq 1}$ is at most $p^{k_{p}}$. Since for each large enough $n$  (so that $v_{p}(p_{n}) =k_{p}$),  the set of elements in $\mathbb{Z}_{p_{n}}$ of order $p^{k'}$ for some $0 \le k' \le k_{p}$,  forms a cyclic group of cardinality $p^{k_{p}}$, the group $T({\mathbb{Z}_{(p_n)}})_{p}$ is a cyclic group of cardinality $p^{k_{p}}$.
\end{proof}

\subsection{Toeplitz subshifts} \label{Sec:BackgroundToeplitz}
We will use classical notions of symbolic dynamics ({\em e.g.} subshift, words, complexity,$\ldots$) and we refer to Section 2.3 in \cite{DonosoDurandMaassPetite} for a presentation  and the notation of these notions.
We assume some familiarity of the reader with the notion of Toeplitz subshift so we review them succinctly. We refer to \cite{Downarowicz:2005} for a survey on this topic.

Let $x=(x_n)_{n\in \mathbb{Z}} \in {\mathcal{A}}^\mathbb{Z}$, where ${\mathcal{A}}$ is a finite alphabet. For an integer $p \geq 1$, we let 
$Per_p(x)= \{ n \in \mathbb{Z}; x_{n}=x_{n+k p}\text{ for all } k\in \mathbb{Z} \}$ be the set of indexes where $x$ is $p$-periodic. The sequence $x$ is said to be \emph{Toeplitz} if there exists a sequence $(p_n)_{n\geq 1}$ in $\mathbb{N}\setminus\{0\}$ such that 
$\mathbb{Z}=\bigcup_{n\geq 1} Per_{p_n}(x)$. Equivalently, the sequence $x$ is Toeplitz if all finite blocks in $x$ appear periodically. We say that $p_n$ is an 
\emph{essential period} if for any $1 \leq p <p_n$ the sets $Per_{p}(x)$ and $Per_{p_n}(x)$ do not coincide. If the sequence $(p_n)_{n\geq 1}$ is formed by essential periods and $p_n$ divides $p_{n+1}$, we call it a \emph{periodic structure} of $x$. Clearly, if $(p_n)_{n\geq 1}$ is a periodic structure, then 
$(p_{i_n})_{n\geq 1}$ is also a periodic structure for any strictly increasing sequence of positive integers $(i_n)_{n\geq 1}$. 

A subshift $(X,\sigma)$ is a \emph{Toeplitz subshift} if $X$ is the orbit closure of a Toeplitz sequence $x \in X$. 
The subshift $(X, \sigma)$ is also referred as the {\em subshift generated} by $x$.
Let  $\textbf{?}$ be a symbol not in $\mathcal{A}$. If $(p_n)_{n\geq 1}$ is a  periodic structure of $x$, then for every $n \geq 1$ we can define the \emph{skeleton map} at scale $p_n$ by $S_{p_n}\colon X\to (\mathcal{A}\cup \text{?})^{\mathbb{Z}}$ by putting  
$(S_{p_n}(y))_{m}$  equal to $y_m$ if $m \in Per_{p_n}(y)$ and to $?$ otherwise. 
Not all the points $y \in X$ are Toeplitz sequences, but they all have the same {\em skeleton structure} $(S_{p_n}(y))_{n\geq 1}$ modulo a shift. More precisely, if $(p_n)_{n\geq 1}$ is the periodic structure of the Toeplitz sequence $x$ and $y$ is any point in $X$, then for any $n\geq 1$ there exists $j_n\in \{0,\ldots,p_n-1\}$ such that $Per_{p_n}(y)=Per_{p_n}(x)-j_n$ and sequences $x$ and $y$ coincide on these coordinates, {\em i.e.}, $S_{p_n}(y)=\sigma^{j_n}S_{p_n}(x)$ (see \cite{Downarowicz:2005}, Section 8).  

It is well known that a minimal subshift $(X,\sigma)$ is a Toeplitz subshift if it is an almost 
one-to-one extension of an odometer. The odometer is given by $\mathbb{Z}_{(p_n)}$, 
where $(p_n)_{n\geq 1}$ is a periodic structure of a Toeplitz point $x$ generating $X$. 
The projection of a point $y \in X$ into the odometer is given by the sequence 
$(j_n)_{n\geq 1}$ described above (see for instance \cite{Williams84}). Hence, the projection of the Toeplitz sequence $x$ is ${\bf 0}$.
Moreover, this odometer is the \emph{maximal equicontinuous factor} of $(X,\sigma)$. 
A finite Toeplitz subshift is generated by a periodic Toeplitz sequence so it can be identified with its also finite associated odometer.

Since a Toeplitz subshift is an almost one-to-one extension of its maximal equicontinuous factor, to study its group of automorphisms, we will use the following result proved in \cite{DonosoDurandMaassPetite} (Lemma 2.1 and Lemma 2.4). 

\begin{lemma}[\cite{DonosoDurandMaassPetite}]\label{lem:AutoExtensiona11} 
	Let $(X,T)$ be a minimal system and $\pi \colon (X,T) \to (Y,S)$ be the projection of $X$ onto its maximal equicontinuous factor. Then, we can define a map  
	$\widehat{\pi}:	{\rm Aut} (X,T)   \to  {\rm Aut} (Y,S)$, 
	$\phi \ \mapsto \ \widehat{\pi}(\phi)$,  
	such that $\widehat{\pi}(\phi)(\pi(x))=\pi(\phi x)$ for every $x\in X$. 
	If $\pi$ is a proximal extension (in particular if $\pi$ is an almost one-to-one extension) 
	then $\widehat{\pi}$ is injective. 
\end{lemma}

It is worth mentioning that the same result holds for endomorphisms and that endomorphisms of an equicontinuous system are automatically automorphisms \cite{Aus63}. It follows that the automorphism group of a Toeplitz subshift can be identified with a subgroup of the associated odometer. Indeed, it is well known that the group of automorphisms of an equicontinuous system is homeomorphic to the space itself (see \cite{Aus63} or Lemma 5.9 in \cite{DonosoDurandMaassPetite} for a shorter proof). It follows that the automorphism group of a Toeplitz subshift is a countable abelian group.  

Many other results described below will follow from the analysis of subgroups of odometers. As a first remark we have that if $(X, \sigma)$ is a Toeplitz subshift and $\mathbb{Z}_{(p_n)}$ is its associated odometer, then any finitely generated subgroup $G \le  {\rm Aut}(X, \sigma)$ is isomorphic to $\mathbb{Z}^d \oplus H$, where, by Lemma \ref{lem:torsion}, $H$ is a finite cyclic abelian group. If the odometer has no torsion, then the group $H$ is trivial and thus $G$ is isomorphic to $\mathbb{Z}^d$.
This kind of properties restrict the groups that can be realized as the automorphism groups of  Toeplitz subshifts. We already mentioned in the introduction  that the group of rational numbers with the addition cannot be injected in any odometer $\mathbb{Z}_{(p_n)}$. Nevertheless,
we will see in Section \ref{sec:realisation} that any finitely generated abelian group 
whose torsion subgroup is cyclic can be realized as the automorphism group of a Toeplitz subshift.

\subsection{Automorphism group of disjoint Toeplitz subshifts}
Two topological dynamical systems $(X,T)$ and $(Y,S)$ are said to be \emph{disjoint} if the product system $(X\times Y,T\times S)$ does not have any non-empty, closed and $T\times S$ invariant subsets projecting onto $X$ and $Y$ respectively, different from $X\times Y$. 

In what follows we use the symbol $\oplus$ (instead of $\times$) whenever we want to stress that a product is in the group category. We start with a general lemma.
Notice that the inclusion ${\rm End}(X,T)\oplus {\rm End}(Y,S) \subseteq {\rm End}(X\times Y,T\times S)$ is always true.

\begin{lemma} \label{Lemma:AutoProduct}
	Let $(X,T)$ and $(Y,S)$ be disjoint minimal systems. If $ \phi \in{\rm End}(X\times Y,T\times S)$ (resp. ${\rm Aut}(X\times Y,T\times S)$) commutes with  ${\rm id} \times S$ and $T\times {\rm id}$, then $\phi \in {\rm End}(X,T)\oplus {\rm End}(Y,S)$ (resp. $ {\rm Aut}(X,T)\oplus {\rm Aut}(Y,S)$). In particular, the conclusion holds if  ${\rm End}(X\times Y,T\times S)$ is abelian. 
\end{lemma}
\begin{proof}
	Write the  endomorphism $\phi(x,y)=(\phi_1(x,y),\phi_2(x,y))$. If $\phi$ commutes with ${\rm id} \times S$ and $T\times {\rm id}$, then we get that $\phi_1(x,y)=\phi_1(x,S^n y)$ and $\phi_2(T^nx,y)=\phi_2(x,y)$ for every $n\in \mathbb{Z}$. By minimality of $(X,T)$ and $(Y,S)$ we get that  $\phi_1$ only depends on $x$ and $\phi_2$ only depends on $y$, meaning that $\phi$ belongs to ${\rm End}(X,T)\oplus {\rm End}(Y,S)$. The same is true for $\phi$ an automorphism.  
\end{proof}

\begin{corollary} \label{Corollary:AutoProductToeplitz}
	Let $(X_1,\sigma)$ and $(X_2,\sigma)$ be disjoint Toeplitz subshifts. Then, ${\rm End}(X_1 \times X_2,\sigma \times \sigma) = {\rm End}(X_1, \sigma)\oplus {\rm End}(X_2,\sigma)$  and ${\rm Aut}(X_1 \times X_2,\sigma \times \sigma) = {\rm Aut}(X_1, \sigma)\oplus {\rm Aut}(X_2,\sigma)$. In particular, if $(X_1,\sigma)$ and $(X_2,\sigma)$ are coalescent then $(X_1\times X_2,\sigma\times \sigma)$ is coalescent too.
\end{corollary}
\begin{proof}
	It suffices to notice that the system $(X_1 \times X_2, \sigma\times \sigma)$ is a Toeplitz subshift itself. By Lemma \ref{lem:AutoExtensiona11}, ${\rm End}(X_1 \times X_2,\sigma \times \sigma)$ is abelian and we can apply Lemma \ref{Lemma:AutoProduct}. 
\end{proof}

\section{Automorphism group of Toeplitz subshifts with subquadratic complexity}
\label{Sec:ToeplitzSubquadratic}

In this section we study the automorphism groups of Toeplitz subshifts of subquadratic complexity. That is, the complexity function verifies $\displaystyle \liminf_{n\to \infty}\frac{p_X(n)}{n^2}=0$. In this case, a simple argument relying in  Lemma 5 of \cite{QZ} implies that these systems are coalescent. In fact, if $(X,\sigma)$ is a minimal subshift of subquadratic complexity (not necessarily Toeplitz), we can consider the {\em spacetime} tiling of an endomorphism $\phi$, as done in \cite{CyrKra, CyrKra4,CyrKra5}, 
and obtain a periodicity condition on this tiling that is translated into $\phi^n=\sigma^m$ for some $n\in \mathbb{N}$ and $m\in \mathbb{Z}$. From this we deduce that $\phi$ is injective and then is an automorphism. 

Recall that for an abelian group $G$, $T(G)$ is its torsion subgroup and that $G/T(G)$ is a torsion free group. Most of the proofs rely on the following property of odometers.

\begin{lemma}\label{lem:Algebre} If $G$ is an abelian group and ${\bf s} \in G$ is an element of infinite order  
	such that $G/\langle {\bf s}, T(G) \rangle$ is finite, then $G/T(G)$ is a cyclic group isomorphic to $\langle {\bf s} \rangle$. In particular, the quotient  $\left(G/T(G) \right)/\langle {\bf s} \rangle$ is also a cyclic group. 
\end{lemma}

\begin{proof}  
	Since $G \slash \langle {\bf s}, T(G)\rangle$ is finite, let  $g_{1}, \ldots, g_{m} \in G$ be representatives  for all of its cosets. For every $i \in \{1, \ldots, m\}$ there exists an integer $\ell_{i} \in \mathbb{Z}$ such that $\ell_{i}g_{i} \in \langle {\bf s}, T(G) \rangle$.  Let $\ell$ denote the smallest positive integer such that $\ell g \in \langle {\bf s}, T(G)  \rangle$ for every $g \in G$. It is standard to check that $\ell$ divides  $lcm(\ell_{1}, \ldots, \ell_{m})$.

	Now, for each $i\in \{1,\ldots, m\}$, let $k_{i}$ be the integer such that $\ell g_{i} \equiv  k_{i}{\bf s} \mod T(G)$. Since $G/T(G)$ is torsion free,  the minimality of $\ell$ gives that \\ 
	$(\ell, k_{1}, \ldots, k_{m}) =1$. By Bezout's Theorem there exist integers $a_{0}, \ldots, a_{m}$ such that 
	$$a_{0}\ell + a_{1}k_{1} + \cdots + a_{m}k_m =1.$$

	Therefore, ${\bf s} = (a_{0}\ell + a_{1}k_{1} + \cdots + a_{m}k_m){\bf s}\equiv \ell(a_{0} {\bf s}  + a_{1}g_{1} + \cdots + a_{m}g_m) \mod T(G)$, and consequently there exists $g \in G$ such that $\ell g \equiv {\bf s} \mod T(G)$.  Since $\ell g_{i} \equiv  \ell k_{i}  g \mod T(G)$ for each $i \in \{ 1, \ldots, m\}$ and $G/T(G)$ is torsion free,  it follows $g_{i} \equiv k_{i}g \mod T(G)$. This together with $\ell g \equiv {\bf s} \mod T(G)$ shows that $G/T(G)$ is generated by $g$, thus is cyclic. Clearly $(G/T(G) )\slash \langle {\bf s} \rangle$ is also cyclic.  
\end{proof}

\begin{theorem} \label{Thm:RootsToeplitz}
	Let $(X,\sigma)$ be a Toeplitz subshift and let $T=T({\rm Aut}(X,\sigma))$ be the torsion subgroup of ${\rm Aut}(X,\sigma)$.
	\begin{enumerate}
		\item \label{item:Subquadratic} If ${\rm Aut}(X,\sigma)/\langle T, \sigma  \rangle $ is a periodic group, then any subgroup $H$ of ${\rm Aut (X,\sigma)}$ containing the shift is spanned by the roots of $\sigma \mod T(H)$ in $H$ and $T(H)$. In particular, a torsion free subgroup  $H$ containing $\sigma$ is spanned by the roots of $\sigma$ in $H$.
		\item \label{item:NonSuperlinear} If ${\rm Aut}(X,\sigma)/\langle  \sigma  \rangle $ is finite, then ${\rm Aut}(X,\sigma)$ is isomorphic to $\mathbb{Z} \oplus T$   and $T$ is either trivial or isomorphic to some  $\mathbb{Z}_{N}$.
	\end{enumerate}
\end{theorem}
The statement \eqref{item:Subquadratic} of this theorem applies for instance when $(X,\sigma)$ has subquadratic complexity, by the main result in \cite{CyrKra}.  Statement \eqref{item:NonSuperlinear} applies when $(X,\sigma)$ has non-superlinear complexity \cite{DonosoDurandMaassPetite}. 
Moreover, since ${\rm Aut}(X, \sigma)$ embeds into the odometer $\mathbb{Z}_{(p_{n})}$ associated to $(X, \sigma)$, Lemma \ref{lem:torsion} implies that  any  prime divisor $p$ of $N$    satisfies  $\lim_{n\to\infty} v_{p}(p_{n})$ is positive and finite. 
In this situation, the odometers of the Toeplitz substitutions considered in \cite{CovenQuasYassawi} are groups of $p$-adic integers $\mathbb{Z}_{(p^n)}$, that have no torsion. So it follows that in this case ${\rm Aut }(X,\sigma)$ is itself a cyclic group, which corresponds to Corollary 12 in \cite{CovenQuasYassawi}.

\begin{proof}
	Let $\mathbb{Z}_{(p_n)}$ be the odometer associated to the Toeplitz subshift $(X,\sigma)$ and let  $H$ be a subgroup of ${\rm Aut}(X,\sigma)$ containing $\sigma$. 
	For $\phi \in H$, consider the subgroup $G$ of $H$ spanned by $\phi$ and $\sigma$. 
	By Lemma \ref{lem:AutoExtensiona11} we can see $G$ as a subgroup of the odometer $\mathbb{Z}_{(p_n)}$ and translate the hypothesis to the statement: $G/\langle T(G), {\bf 1} \rangle$ is an abelian finitely generated periodic group. But this implies that  $G/\langle  T(G) , {\bf 1} \rangle$ is in fact finite.  From Lemma \ref{lem:Algebre}, the group  $G\slash T(G)$ is cyclic. In particular, there exist $\rho \in G$ and $m_1,m_2 \in \mathbb{Z}$ such that $\phi \equiv \rho ^{m_1} \mod T(H)$ and $\sigma=\rho^{m_2} \mod T(H)$. This completes \eqref{item:Subquadratic}.
	
	If ${\rm Aut}(X,\sigma)/\langle  \sigma  \rangle $ is finite, then ${\rm Aut}(X,\sigma)$ is finitely generated. So Lemma \ref{lem:torsion} implies that its torsion subgroup $T$ is cyclic. 
	A direct consequence of  Lemma \ref{lem:Algebre} to ${\rm Aut}(X, \sigma)$ is that ${\rm Aut}(X,\sigma)/ T$ is cyclic.  Since ${\rm Aut}(X,\sigma) \simeq  \left({\rm Aut}(X,\sigma)/ T \right) \oplus T$ we have proved statement \eqref{item:NonSuperlinear}. 
\end{proof}

It is worth noting that if $\varphi$ is a root of $\sigma$, then the integer $\ell$ such that $\varphi^\ell=\sigma$ has to be prime with each $p_n$ appearing in the odometer $\mathbb{Z}_{(p_{n})}$. This follows from the fact that the equality 
$\varphi^\ell=\sigma$ (in ${\rm Aut}(X,\sigma)$) is translated into $\ell z={\bf 1}$ for $z \in \mathbb{Z}_{(p_{n})}$, which is possible only if $\ell$ is prime with $p_n$ for all $n\geq 1$. 
This observation allows us to produce examples of Toeplitz subshifts without roots, so where the automorphism group is trivial.

We illustrate Theorem \ref{Thm:RootsToeplitz} and previous comment in the next general construction. We concentrate on part \eqref{item:NonSuperlinear}. Examples for part \eqref{item:Subquadratic} appear in the next section. 

Consider a sequence $(w_n)_{n\geq 1}$ on the finite alphabet $\mathcal{A}\cup \{ \text{?} \}$ such that for each $n\geq 1$, $|w_n| = q_n\geq 3$ and, for $n\geq 2$, $w_n=u_n?v_n$, where $u_n$ and $v_n$ are non empty words on the alphabet $\mathcal{A}$. Thus $w_n$ contains exactly one  symbol $?$.
Now define the sequence $(W_n)_{n\geq 1}$ by: $W_1=w_1 ^\infty=\ldots w_1w_1.w_1w_1\ldots$, where the central dot indicates the position to the left of the zero coordinate, and 
$W_{n+1}=F_{W_n}(w_{n+1}^{\infty})$ for every $n\geq 1$. Here, $F_{W_n}(w_{n+1}^{\infty})$
is the sequence obtained from $W_n$ replacing consecutively all the symbols $?$ by the sequence 
$w_{n+1}^{\infty}$, where $(w_{n+1}^{\infty})_0$ is placed in the first $?$ to the right of  coordinate $0$. The map $F$ will be studied in more details in next section. 

Since symbol $\text{?}$ moves away from zero coordinate with $n$, then the sequence $(W_n)_{n\geq 2}$ converges to a point $x\in\mathcal{A}^\mathbb{Z}$. In addition, each coordinate of $x$ is periodic with periods in the sequence $(p_n)_{n\geq 1}$, where $p_n=q_1\cdots q_n$. Hence $x$ is a Toeplitz sequence. We let the reader check that one can choose the words $u_n$'s and $v_n$'s to construct a non periodic sequence $x$. Also, special choices of the $u_n$'s and $v_n$'s allow us to prove that $(p_n)_{n\geq 1}$ is the sequence of essential periods of $x$ (for example consider two different letters  $a,b \in {\mathcal A}$ and take $u_n=a$, $v_n=b^{q_n-2}$). So, under this assumption we have that 
$\mathbb{Z}_{(p_n)}$ is the odometer associated with $X$, the orbit closure of $x$ by the shift map.
Finally, remark that any word of length $p_n$ appearing in $x$ can be constructed filling one symbol $\text{?}$ in $W_n$, which is a periodic sequence of period $p_n$. Then, the complexity of $X$ verifies $p_X(p_n)\leq |{\mathcal A}| \ p_n$, and thus $X$ has non-superlinear complexity. 

Fixing the values of $(q_n)_{n\geq 1}$ in such a way that $p_{n}=n!$ for all $n\geq 3$ we get a Toeplitz subshift whose odometer is the universal one. In this case, by Lemma \ref{lem:torsion} we have that ${\rm Aut}(X,\sigma)$ is torsion free and by Theorem \ref{Thm:RootsToeplitz} it is spanned by the roots of $\sigma$. But by the discussion after previous theorem the only possible root is the shift itself. Again, by Lemma \ref{lem:torsion}, one can make other choices of the sequence $(q_n)_{n\geq 1}$ in such a way that ${\rm Aut}(X,\sigma)$ is torsion free and isomorphic to $\mathbb{Z}$. 

The same construction together with Corollary \ref{Corollary:AutoProductToeplitz} allow us to get Toeplitz subshifts with non-superlinear complexity such that ${\rm Aut}(X,\sigma)$ is isomorphic with  $\mathbb{Z} \oplus \mathbb{Z}_N$ for any $N\in \mathbb{N}$. Indeed, in previous construction consider $p$ a prime number not dividing $N$ and set $q_n=p$ for any $n\geq 1$. Then consider the cartesian product system 
$(X\times \mathbb{Z}_N,\sigma\times +{\bf 1})$, which, by the mentioned corollary, is a Toeplitz subshift with the desired automorphism group.

As consequence of the discussion of this section we can formulate the following dichotomy. 

\begin{corollary} \label{Cor:DichotomyRoots}
	Let $(X,\sigma)$ be a Toeplitz subshift such that ${\rm Aut}(X,\sigma)/\langle \sigma \rangle$ is a periodic group. If ${\rm Aut}(X,\sigma)$ is torsion free, then either ${\rm Aut}(X,\sigma)$ is cyclic or ${\rm Aut}(X,\sigma)$ is not finitely generated. 
\end{corollary}
\begin{proof}
	The dichotomy follows from considering the cases whether ${\rm Aut}(X,\sigma)$ is spanned by finitely or infinitely many roots of $\sigma$.
\end{proof}
We can think of this result as a consequence of the algebraic structure of the automorphism group of a Toeplitz subshift, which for given automorphisms $\phi_1$ and $\phi_2$ allow us to take a ``common divisor'', i.e. an automorphism $\phi$ such $\phi_1$ and $\phi_2$ belong to $\langle \phi \rangle$ (assuming that there is no torsion).  This situation has to be contrasted with the case of a mixing shift of finite type, where the shift always admits only finitely many roots (see the discussion after Problem 3.5 in \cite{BoyleLindRudolph}).

 In the next section we exhibit some examples where the shift does have infinitely many roots. It remains open whether we can give a precise description of the roots of the shift map generating the automorphism groups in previous theorems.

\section{Not finitely generated automorphism groups for Toeplitz subshifts}
\label{sec:pqtoeplitz}

As proved in Theorem \ref{Thm:RootsToeplitz}, 
the quotient of the automorphism group by the torsion subgroup for a Toeplitz subshift of subquadratic complexity is generated by the roots of the shift map. In the particular case of non-superlinear complexity the automorphism group modulo its torsion subgroup is virtually $\mathbb{Z}$ and generated by a unique root. In this section we prove that the last result is in some sense optimal. That is, we can construct Toeplitz subshifts of subquadratic complexity such that the automorphism group is even not finitely generated. The same construction allows to get Toeplitz subshifts of arbitrary polynomial complexity whose automorphism groups are not finitely generated. 

The main result of the section is the following, 

\begin{theorem} \label{Thm:LowComplexityToeplitz}
	For every $\epsilon>0$, there exists a Toeplitz subshift $(X,\sigma)$ such that the complexity verifies $\frac{p_X(n)}{n}\leq C n^{\epsilon}$ for all $n \geq 1$ and whose automorphism group is torsion free and not  finitely generated. 
\end{theorem} 

To achieve this, we make use of the class of $(p,q)$-Toeplitz subshifts. In \cite{CK}, Cassaigne and Karhum\"aki introduced this class and established the fundamental properties we discuss in the sequel. In this class one can get complexities that are arbitrarily close to non-superlinear (but always superlinear). This class was implicitly used by Salo \cite{Salo} to give an example of a subshift of complexity $p_X(n)\leq Cn^{1.757}$ with a non-finitely generated automorphism group. We simplify and extend his result in the generality of Theorem \ref{Thm:LowComplexityToeplitz}. We start introducing the basic notions from \cite{CK}, then we use freely their results. It is worth noting that in \cite{CK} the construction was carried out for one sided subshifts, but all can be extended without any problem to the two sided case.

\subsection{$(p,q)$-Toeplitz subshifts}
We refer to Section 2 in \cite{CK} for a detailed discussion on next properties and concepts. 

Let $\mathcal{A}$ be a finite alphabet and $?$ a letter not in $\mathcal{A}$ (usually the symbol $?$ is referred as a ``hole'').
Let $x$ $\in (\mathcal{A}\cup \{ \text{?} \})^{\mathbb{Z}}$. The sequence $x$ represents a sequence over the alphabet $\mathcal{A}$ with holes. Given $x, y \in (\mathcal{A}\cup \{\text{?}\})^{\mathbb{Z}}$, define 
$F_x(y)$ as the sequence obtained from $x$ replacing consecutively all the $?$ by the symbols of $y$, where $y_0$ is placed in the first $?$ to the right of coordinate $0$. In particular, if 
$x$ has no holes, $F_{x}(y)=x$ for every $y \in (\mathcal{A}\cup \{\text{?}\})^{\mathbb{Z}}$. 
In addition, observe that:
\begin{equation}\label{eq:magicformula}
\text{ if } \ z=F_{x}(y) \text{ then } \ F_{z}=F_{x}\circ F_{y}.
\end{equation}

Now, consider a finite word $w$ in $\mathcal{A}\cup \{ \text{?}\}$.  
Let $p$ be the length of $w$ and $q$ the number of its holes.
Denote by $w^\infty$ the sequence $\cdots www.www\cdots \in (\mathcal{A}\cup \{ \text{?} \})^{\mathbb{Z}}$, where the central dot indicates the position to the left of coordinate $0$. 
We define the sequence $(T_n(w))_{n\geq 1}$ by: $T_1(w)=w^\infty$ and 
$T_{n+1}(w)=F_{w^{\infty}}(T_n(w))$ for every $n\geq 1$. It is not complicated to see that each 
$T_n(w) = u_{n}^{\infty}$  for some word $u_{n}$ of length $p^n$ and  $u_{n}$ has $q^n$ holes.
We have that the limit $\displaystyle x=\lim_{n\to \infty} T_n(w)$ is well defined as a sequence in $(\mathcal{A}\cup \{\text{?}\})^{\mathbb{Z}}$. Moreover, if $w$ does not start or finish with a hole, then the limit sequence belongs to $\mathcal{A}^{\mathbb{Z}}$, {\em i.e.}, $x$ has no holes. 
The point $x$ is called a $(p,q)$-{\em Toeplitz sequence}.
Its orbit closure under the shift map $X$ is called a $(p,q)$-{\em Toeplitz subshift}. 

One of the main results in \cite[Theorem 5]{CK}  states that the complexity of a non periodic
$(p,q)$-Toeplitz is $\Theta(n^r)$, where $r=\log(p/d)/\log(p/q)$ and $d=(p,q)$.

Now suppose that $(p,q)=1$.
Then, there exist positive constants $C_1$ and $C_2$ such that 
\begin{equation}\label{eq:complexity}
C_1 n^{1+\frac{\log q}{\log p -\log q}} \leq p_X(n)\leq C_2 n^{1+\frac{\log q}{\log p -\log q} } ,\quad \forall n\geq 1.
\end{equation}
Moreover, the length of $|u_n|$ above is the smallest possible, $(p^n)_{n\geq 1}$ is a periodic structure for $x$ and the associated odometer is given by $\mathbb{Z}_{(p^n)}$ which is torsion free by Lemma \ref{lem:torsion}. 
Therefore, thanks to Lemma \ref{lem:AutoExtensiona11} we get the same conclusion for ${\rm Aut}(X,\sigma)$. 

If $p$ is large enough compared with $q$, we have that $|\frac{\log q}{\log p -\log q}|\leq \epsilon$. Thus, $(p,q)$-Toeplitz subshifts is a natural class to study to prove Theorem \ref{Thm:LowComplexityToeplitz}.

\subsection{Self reading properties of $(p,q)$-Toeplitz}
Let $w$ be a finite word in  $(\mathcal{A}\cup \{?\})$ of length $p$ and $q$ holes with $(p,q)=1$. Let $x=\lim_{n\to \infty} T_n(w)$ be the $(p,q)$-Toeplitz sequence generated and $X$ its orbit closure under the shift. The purpose of this section is to prove the main consequences of the 
so called ``self-reading'' property of $x$. 
First we summarize some basic results, some were already discussed in \cite{CK}, others need to be proved. 

\begin{proposition}\label{prop:basicproperties} We have, 
	\begin{enumerate} 
		\item \label{item:fixpoint} For every $n\geq 1$, the map $F_{T_n(w)}\colon (\mathcal{A}\cup \{?\})^{\mathbb{Z}}\to (\mathcal{A}\cup \{?\})^{\mathbb{Z}}$ is continuous and $F_{{T_n(w)}}(x)=x$.  
		\item \label{item:sigmapq} $F_{{T_n(w)}}(\sigma^{q^n} y)=\sigma^{p^n}F_{{T_n(w)}}(y)$ for every $y\in (\mathcal{A}\cup \{?\})^{\mathbb{Z}}$.
		\item \label{item:MaxEqFactor} The skeleton structure of $x$ is $(T_n(w))_n$, {\em i.e.,} $S_{p^n}(x) = T_{n}(w)$ for all $n\geq 1$. 
		\item \label{item:Minimalq^i} The transformation $\sigma^{q^n}$ is minimal in $X$ for every $n\geq 1$.
		\item \label{item:F_W} $F_{T_n(w)}$ leaves $X$ invariant. 
	\end{enumerate}
\end{proposition}
\begin{proof}
	Continuity in statement \eqref{item:fixpoint} is direct from definition. Now, from \eqref{eq:magicformula} we have that $T_{n+1}(w)=F_{w^{\infty}}(T_{n}(w))$, and thus $F_{T_{n+1}(w)}=F_{w^{\infty}}\circ F_{T_n(w)}$, which implies $F_{T_{n}(w)}= (F_{w^{\infty}})^{n}$. 
	From this equality we deduce $F_{T_{n}(w)}(x)=\lim_{m\to \infty} (F_{w^{\infty}})^{n+m}(w^{\infty})=x$. This proves \eqref{item:fixpoint}.
	
	Statement \eqref{item:sigmapq} follows from the fact that each $T_n(w)$ is periodic of period $p^n$ and contains $q^n$ holes. Statement \eqref{item:MaxEqFactor} is direct.
	
	Let $\pi$ denote the factor map from $X$ to $\mathbb{Z}_{(p^n)}$. 
	Since $(p,q)=1$, statement (\ref{item:Minimalq^i}) follows from the fact that translation 
	by ${q^n}{\bf 1}$ acts minimally in the odometer $\mathbb{Z}_{(p^n)}$ by Lemma \ref{lem:generalizedminimality}, for every $n\geq 1$. Indeed, if $A \subseteq X$ is closed and invariant under $\sigma^{q^n}$, 
	then $\pi(A)$ is also closed (since $X$ is compact) and invariant.  By minimality we get that $\pi(A)=\pi(X)=\mathbb{Z}_{(p^n)}$. But $\pi$ is almost one-to-one and thus $A$ contains all points with one preimage (which is a $G_{\delta}$-set). Since $A$ is closed we get $A=X$. This proves (\ref{item:Minimalq^i}).
	
	We finally notice that properties \eqref{item:fixpoint} and \eqref{item:sigmapq} 
	imply that $F_{{T_n(w)}}((\sigma^{q^n})^m x) \in X$ for all $n\geq 1$ and $m\in \mathbb{Z}$. The minimality of $\sigma^{q^n}$ implies that $F_{{T_n(w)}}$ leaves invariant $X$, proving (\ref{item:F_W}).  
\end{proof}

\subsection{Proof of Theorem \ref{Thm:LowComplexityToeplitz}}

First, we choose a word $w$ of length $p$ and with $q$ holes in $\mathcal{A}\cup\{ ?\}$ such that the associated limit sequence $x$ is  non periodic and in $\mathcal{A}^\mathbb{Z}$. It is enough to avoid the symbol $?$ in the first and last coordinates of $w$. Recall we are considering $p$ and $q$ to be relatively primes. 
Notice that, by \eqref{eq:complexity}, within the family of $(p,q)$-Toeplitz subshifts with $p$ and $q$ relatively primes we can get for any $\epsilon >0$ complexities such that for some constant $C>0$, ${p_X(n)}/{n}\leq C n^{\epsilon}$, for all $n \geq 1$. It remains to prove that the automorphism group of the subshift $X$ generated by $x$ is not finitely generated. 

Let $n \geq 1$ and $z \in X$. As observed in Section \ref{Sec:BackgroundToeplitz}, we can find  \ a unique $m=m(n,z)\in \{0,\ldots,p^n-1\}$ such that $S_{p^n}(z)=\sigma^{m}S_{p^n}(x)=\sigma^{m}T_n(w)$, where in the last equality we have used Proposition \ref{prop:basicproperties} \eqref{item:MaxEqFactor}. Recall that $S_{p^n}(\cdot)$ is the skeleton map at scale $p^n$ and $m$ is nothing but the projection of $z$ onto the factor $\mathbb{Z}_{p^n}$. 

Let $z \in X \mapsto H_n(z)$ be the map that associates  the sequence in ${\mathcal A}^{\mathbb{Z}}$ of the consecutive symbols of $z$ 
in the holes of $S_{p_n}(z)$ from the coordinate $-m(n,z)$. 
That is, the unique sequence satisfying 
\begin{equation}
\label{eq:HN}
F_{S_{p^n} (x)} (H_n (z)) = \sigma^{-m(n,z)} (z) .
\end{equation}
We have that $H_n$ is a continuous function.
Let $k$ be an integer and $r$ such that $0\leq r < p^n$  and $k+m=lp^n + r$ for some  $l\in \mathbb{Z}$. 
It is straightforward to check that $m(n,\sigma^k (z))=r$. 
Then, using \eqref{eq:HN} and Proposition \ref{prop:basicproperties} \eqref{item:sigmapq}, one gets
\begin{equation} \label{PropertiesH}
H_n(\sigma^{k} z)= \sigma^{lq^n}H_{n}(z) .                                             
\end{equation}
We claim that $H_n(z)$ belongs to  $X$.
Since $S_{p^n}(z) = \sigma^mS_{p^n}(x)$, there is an integer sequence $(\ell_{i})_{i\ge 0}$ so that  $z$ can be written as $z=\lim_{i\to \infty} \sigma^{p^n \ell_i+m}(x)$.
Using formula \eqref{PropertiesH} we get, 
$$H_n(z)=H_n(\lim_{i\to \infty} \sigma^{p^n \ell_i+m}(x))=\lim_{i\to \infty}
\sigma^{q^n \ell_i} H_n(\sigma^{m}(x)).$$
But, from \eqref{eq:HN}, $H_n(\sigma^{m}(x))=H_n(x)$ and 
by Proposition \ref{prop:basicproperties} \eqref{item:fixpoint} $H_n(x)=x$, then 
$H_n(z)=\lim_{i\to \infty} \sigma^{q^n \ell_i}(x) \in X$. The claim is proved. 

Define $\varphi_{n}(z)=\sigma^{m}F_{S_{p^n}(x)}(\sigma(H_{n}(z)))$.  We have that this map $\varphi_n: X \to X$ is well defined by the previous claim and Proposition \ref{prop:basicproperties} \eqref{item:F_W}. This map leaves invariant each ${\mathcal A}$-letter in the $p^n$-skeleton of $z$ and shift the symbols of $z$ in the holes by one within the holes.

Using again formula \eqref{PropertiesH} it is not difficult to check that $\varphi_n$ is an automorphism of $(X,\sigma)$. Moreover, we have that $\varphi_n^{q^n}=\sigma^{p^n}$ because there are $q^n$ holes in the first $p^n$ letters of $S_{p^n}(x)=T_n(w)$.

On the other hand, $q^n$ is the minimum positive integer $\ell$ such that $\varphi_n^{\ell}\in \langle \sigma \rangle$. 
Indeed, if $\varphi_n^{\ell}=\sigma^{r}$ for some integer $r$, then $\sigma^{\ell p^n}=\varphi_n^{\ell q^n}=\sigma^{rq^n}$ and by aperiodicity, $\ell p^n=r q^n$. 
Since $p$ and $q$ are relatively  prime,  $ q^n$ divides $\ell$.

In particular, $\left \langle \{\varphi_n: n\in \mathbb{N} \}, \sigma \right \rangle/\langle \sigma \rangle$ is an infinite periodic group and thus it is not finitely generated (finitely generated torsion abelian groups are finite). This implies that ${\rm Aut}(X,\sigma)$ is not finitely generated. 
$\square$
\smallskip

We finish this section pointing out that if we choose $p$ and $q=p-2$ for a large odd value of $p$, then the associated $(p,q)$-Toeplitz subshift has a polynomial complexity of degree ${\frac{\log(p)}{\log(p)-\log(p-2)}\geq \frac{p-2}{2} \log(p)}$. We conclude that,

\begin{corollary}\label{cor:arbitrary}
	There exist Toeplitz subshifts of arbitrarily large polynomial complexity with a not finitely generated automorphism group.
\end{corollary}

In the light of Corollary \ref{Cor:DichotomyRoots}, we notice that we can easily exhibit infinitely many roots of the shift. 
\begin{corollary}
	Let $(X,\sigma)$ be a $(p,q)$-Toeplitz, with $(p,q)=1$. Then for every $n\geq 1$, $\sigma$ admits a $q^n$-root.	
\end{corollary}

\begin{proof}
	We have that for every $n\geq 1$ there exists an automorphism $\varphi_n$ such that $\varphi_n^{q^n}=\sigma^{p^n}$. Since $(p,q)=1$ we can find $a,b\in \mathbb{Z}$ such that $ap^n=bq^n+1$. Then $\varphi_n^{aq^n}=\sigma^{ap^n}=\sigma^{bq^n}\sigma$ and the automorphism $\varphi_n^a\sigma^{-b}$ is a $q^n$-root of $\sigma$.    	
\end{proof}

\section{Realization of  finitely generated abelian groups}\label{sec:realisation}

In this section we show that within the class of Toeplitz subshifts we can realize any finitely generated abelian group with cyclic torsion subgroup as an automorphism group. Recall that the property of the torsion is necessary by Lemma \ref{lem:torsion}. In the process, we show that large entropy does not suffice to have a large automorphism group, by constructing Toeplitz subshifts with arbitrarily large entropy and no automorphisms other than powers of the shift. This result is a consequence of the following theorem.

\begin{theorem}\label{Thm:EntropyToeplitz2}
	For any infinite odometer there exists a uniquely ergodic Toeplitz subshift $(X, \sigma)$ with an arbitrarily large topological entropy whose associated odometer is equal to the given one and 
	${\rm Aut}(X,\sigma) = {\rm End}(X,\sigma) =\langle \sigma \rangle$.  
\end{theorem}
\begin{proof}
	Fix an odometer $\mathbb{Z}_{(p_{n})}$ not isomorphic to a finite group. In this case we can consider $(p_{n})_{n\geq 1}$ strictly increasing. Below we construct a Toeplitz point $x$ and its associated subshift $X$ in an iterative process and prove all the desired properties. Without loss of generality we may assume that $p_{1} =1$.
	
	\noindent {\it The Toeplitz subshift.} Let $D_{0}>1$ and $k_{1}>3$ be  constants that will be adjusted later.  In the following, we will assume that $k_{1}$ is large enough so that $2^{n-2} k_{1} > (n+1)^2D_{0}$ for all integer $n \ge 2$. 
	We consider an alphabet  $\mathcal{A}$ with $k_{1}$ letters   and let $i_{1}=1$.
	
	Fix $n\geq 2$ and suppose that at step $(n-1)$ we have defined $k_{n-1} \ge 2^{n-2} k_{1}$ words 
	$B_{1,n-1},\ldots,B_{k_{n-1},n-1}$ of the same length $p_{i_{n-1}}$ on  $\mathcal{A}$. 
	Pick a positive integer $i_{n}>i_{n-1}$ such that $p_{i_{n}} >  p_{i_{n-1}}  3k_{n-1} ( (n^2 D_{0})^{-1}  - k_{n-1}^{-1})^{-1}$ (the definition of $k_{1}$ ensures
	that the term $(n^2 D_{0})^{-1}  - k_{n-1}^{-1}$ is positive).
	Next, we build words of length $p_{i_{n}}$  by concatenating the words  $B_{1,n-1},\ldots,B_{k_{n-1},n-1}$ according to the following rules:
	\smallskip
	
	\noindent $C1)$  The  words  $B_{1,n-1}\ldots B_{\lfloor k_{n-1}/2 \rfloor,n-1}$ and
	$B_{\lfloor k_{n-1}/2 \rfloor +1 ,n-1} \ldots B_{k_{n-1},n-1}$ respectively always appear as prefix and suffix of all words of step $n$. 
	\smallskip
	
	\noindent $C2)$ After ensuring $C1)$, we complete the remaining positions with $p_{i_{n}}/p_{i_{n-1}}-k_{n-1}$
	words of length $p_{i_{n-1}}$ of the previous step. To do so, we consider all different concatenations of $B_{2,n-1}, \ldots,$ $B_{k_{n-1},n-1}$ (we exclude $B_{1,n-1}$) such that each  $B_{i,n-1}$ appears the same number of times $d_{i,n-1}$, for every $i \ge 2$. 
	More precisely, set $d_{n}=d_{i,n-1}=\lfloor({p_{i_{n}}}/ {p_{i_{n-1}}}  - k_{n-1})/ (k_{n-1}-1) \rfloor$
	for $i\ge 3$ (so it is constant for such $i$) and $\hat d_{n}=d_{2,n-1} = p_{i_{n}}/p_{i_{n-1}} -k_{n-1} -(k_{n-1}-2)d_{n}$. 
	Clearly, $\hat d_{n} \geq d_{n}$. For any such concatenation we obtain a word of step $n$ of length $p_{i_n}$ and we let $k_{n}$ denote the number of words of this length we get. We have the bound  $k_{n} \ge 
	f(d_{n},k_{n-1}-1)$, where 
	$f(d,k) = \frac{(dk)!}{(d !)^k}$ denotes the number of partitions of a set of cardinality $d k$ in $k$ atoms such that each atom has  $d$ elements. 
	Observe that $f(d,k) \ge k!$ and so $k_{n} \ge   2 k_{n-1} \ge 2^{n-1} k_{1}$ for each $n \ge 2$. 
	We choose an order $B_{1,n}, \ldots, B_{k_{n},n}$ of them and continue with step $n+1$. 
	
	By construction, there exists a sequence $x$ in ${\mathcal A}^\mathbb{Z}$ such that   
	for every $n\geq 1$ the word around coordinate 0 is equal to $B_{k_n,n}.B_{1,n}$, where the dot indicates the position to the left of coordinate 0 in $x$. It is not difficult to see that for every $n\ge 1$ the sequence $x$ is a concatenation of words $B_{\cdot, n}$. Moreover, any finite word in $x$ appears periodically with period $p_{i_n}$ for some $n\geq 1$. Indeed, by construction, any finite word in $x$ is a subword of  
	$B_{k_n,n}.B_{1,n}$ for some $n\geq 1$, which by $C1)$ is $p_{i_{n+1}}$-periodic. This implies that $x$ is a Toeplitz sequence. Call $(X,\sigma )$ the associated Toeplitz subshift. 
	
	We will need the following fundamental claim. We say  a word $u$ has a {\em trivial overlapping} with  the word $v$ whenever $u$ appears in $v$ only as a prefix or a suffix, meaning  
	if $v= pus$ for some words $p,s$ then $p$ or $s$ has to be the empty word. 
	
	\begin{claim}\label{claim:overlapping}
		For every $n\geq 1$, each word $B_{i,n}$  has a trivial overlapping with $B_{j,n}B_{k,n}$ for any $i,j,k \in \{1,\ldots,k_n\}$.
		
	\end{claim}
	\begin{proof}
		
		We proceed by induction. The case $n=1$ is true by construction. Now assume the result holds for $n$ and by contradiction assume that $B_{i,n+1}$ is a subword (different from a prefix or a suffix) of $B_{j,n+1}B_{k,n+1}$ for some  $i,j,k \in \{1,\ldots,k_{n+1}\}$.  
		We have that in the word $B_{j,n+1}B_{k,n+1}$, $B_{1,n}$ only appears as prefix of $B_{j,n+1}$ and prefix of $B_{k,n+1}$, otherwise $B_{1,n}$  would be a subword of $B_{j',n}B_{k',n}$ for some $j',k' \in \{2,\ldots,k_n\}$, contradicting the induction hypothesis. But $B_{i,n+1}$ also starts with $B_{1,n}$ so the only possibility is $B_{j,n+1}B_{k,n+1}=B_{i,n+1}B_{k,n+1}$ or $B_{j,n+1}B_{k,n+1}=B_{j,n+1}B_{i,n+1}$, which contradicts our assumption.   	
	\end{proof}
	
	Before giving the rest of the proof, let us fix some notations. For a word $y$ (eventually infinite), $i<j \in \mathbb{Z}$, $y|_{[i,j)}$ denotes the word $y_{i}\cdots y_{j-1}$. 
	We say that  a word $u$ {\em occurs} in $y$ if there exists some $i\in \mathbb{Z}$ such that $y|_{[i, i+|u|)} =u$, where $|u|$ denotes the length of $u$ and the index $i$ is called an {\em occurrence} of $u$.

	Now we check that $X$ is an almost one-to-one extension of the odometer $\mathbb{Z}_{(p_{i_n})}$. 
	For this, it is enough to verify that $(p_{i_n})_{n\geq 1}$ is a periodic structure of $x = (x_{n})_{n}$, and for this it is only left to check that each $p_{i_n}$ is an essential period of $x$. If not, there exists $n \geq 1$ and $1 \leq p <p_{i_n}$ such that $Per_p(x) = Per_{p_{i_n}}(x)$. 
	By condition  $C1)$, $\{0,1,  \ldots, p_{i_{n-1}} -1 \} \subseteq   Per_{p_{i_n}}(x)$, so  $p$ is an occurrence of  $x|_{[0,p_{i_{n-1}})} =  B_{1,n-1}$ in $x$.  Since $B_{1,n}$ is a concatenation of words $B_{\cdot, n-1}$, Claim \ref{claim:overlapping} and condition C2), both  imply that $B_{1,n-1}$ has no  occurrence, except $0$,   in  $B_{1,n} = x|_{[0,p_{i_{n}})}$. This contradicts the fact that $1 \leq p <p_{i_n}$.

	\medskip
	
	\noindent {\it Lower bound for the topological entropy.}
	We use the following lower bound of the topological entropy 
	$\displaystyle h(X,\sigma)= \lim_{n\to \infty }\log p_X(n)/n \ge  \limsup_{n\to \infty} \frac{\log k_{n}}{p_{i_{n}}}$.
	A standard computation using the Stirling formula gives that $$ \displaystyle \log f(d,k)   \geq   (d-1)(k-1) \log (k+1)$$ for every $d$ and $k$ larger than some universal constant. 
	Since we have $d_{n} \ge 3 ( ((n+1)^2 D_{0})^{-1}  - (2^{n-2}k_{1})^{-1})^{-1}-2$ and $k_{n} \ge 2^{n-1}k_{1}$, in what follows we can assume that $k_1$ and $D_0$ are large enough so that $d_{n}$ and $k_{n}$ satisfy the previous estimates for all $n \ge 1$. An iterated use of the  previous inequality leads to:

	\begin{align*} \label{eq.entropie}
	\frac{\log k_{n}}{p_{i_{n}}} & \ge \frac{\log f(d_{n}, k_{n-1}-1)}{p_{i_{n}}} \\
	& \ge   \frac{d_{n}-1}{p_{i_{n}}}  (k_{n-1}-2) \log k_{n-1}   =    \frac{p_{i_{n-1}}}{p_{i_{n}}}  (d_{n}-1) (k_{n-1}-2) \frac{\log k_{n-1}}{p_{i_{n-1}}} \\
	& \ge  \frac{p_{i_{n-1}}}{p_{i_{n}}}  (d_{n}-1)(k_{n-1}-2) \cdots \frac{p_{i_{{1}}}}{p_{i_{2 }}}  (d_{2} -1) (k_{{1}}-2) \frac{\log k_{{1}}}{p_{i_{{1}}}}.
	\end{align*}
	Now, a standard  computation gives that for all $n \geq 2$
	
	$$  \left |\frac{p_{i_{n-1}}}{p_{i_{n}}}  (d_{n}-1)(k_{n-1}-2) -1 \right|  \le\frac 1{k_{n-1}-1}+ \frac{p_{i_{n-1}}}{p_{i_{n}}}3k_{n-1}  \le \frac 1{n^{2} D_{0}}.$$

	Let  $C>0$ and $0< r <1$ be constants such that $z+1\geq \exp(-C|z|)$ when $|z| <r$. 
	Then, if $D_{0}$ is large enough, we can take 
	$z=\frac{p_{i_{n-1}}}{p_{i_{n}}}  (d_{n}-1)(k_{n-1}-2) -1$ to get
	$$\prod_{n\geq 2} \frac{p_{i_{n-1}}}{p_{i_{n}}}  (d_{n}-1)(k_{n-1}-2)  \ge {\rm exp }( -C\sum_{n\geq 2} \frac 1{n^{2}D_{0}}).$$ 
	It follows that 
	$h(X,\sigma) \ge  {\rm exp }( -C\sum_{n\geq 2} \frac 1{n^{2}D_{0}})\frac{\log k_{{1}}}{p_{i_{{1}}}}$. Hence, we can make the entropy arbitrarily large by moving $k_{1}$. 
	
	\medskip
	
	\noindent {\it Unique ergodicity.}
	Condition C2) and Claim \ref{claim:overlapping} both  impose  that for every $n\geq 1$ and $i\in \{1,\ldots,k_n\}$ the set of occurrences of the  word $B_{i,n}$ in $x$ has a specific frequency. More precisely, the word $B_{i,n}$ appears exactly a specific number of time in $B_{j,n+1}$ for every $j  \in \{1,\ldots,  k_{n+1} \}$ (Namely one time if $i=1$ and $d_{i,n}+1$ times otherwise). Since any $x\in X$ is a concatenation of words $B_{\cdot, n+1}$ (Condition C2)), it follows that  the average number of occurrences of $B_{i,n}$,     $\frac{1}{N}\sum_{k=0}^{N-1} 1_{[B_{i,n}]_{0}}(\sigma^k x)$ converges, as $N$ goes to infinity,  to $(1+d_{i,n})/|B_{1,n+1}| = (1+d_{i,n})/p_{i_{n+1}}$ (with the convention $d_{1,n} = 0$). Hence for every ergodic measure $\mu$ we have that $\mu([B_{i,n}]_0)=(1+d_{i,n})/p_{i_{n+1}}$  by the Ergodic Theorem.  Since the cylinders $[B_{i,n}]_0$  (and their images under the powers of the shift) generate the Borel $\sigma$-algebra, the associated subshift is uniquely ergodic. 	
	\medskip
	
	\noindent {\it Automorphism group.}
	We now prove that $(X,\sigma)$ has no other endomorphisms than the powers of $\sigma$. Let $\phi\in {\rm End}(X,\sigma)$ and consider $n\geq 1$ large enough such that 
	$p_{i_{n-1}}$ is greater than the radius of a block map for $\phi$.
	
	By Claim \ref{claim:overlapping} and condition C2), any occurrence  $i \in \mathbb{Z}$  of $B_{1,n}$ is an occurrence of some $B_{j, n+1}$ in $x$. Since $x$ is a concatenation of words $B_{\cdot, n+1}$, the index $i- (k_{n}- \lfloor k_{n}/2 \lfloor )p_{i_{n}} -p_{i_{n-1}}$ is an occurrence of the word 
	$$B_{k_{n-1}, n-1 }B_{\lfloor k_{n}/2 \rfloor +1, n }\ldots B_{k_{n},n} \ B_{1,n} \ B_{2,n} \ldots B_{\lfloor k_{n}/2 \rfloor, n }B_{1,n-1},$$
	meaning that the word $B_{1,n}$ is always preceded in $x$ by $B_{k_{n-1}, n-1 }B_{\lfloor k_{n}/2 \rfloor +1, n }\ldots B_{k_{n},n}$ and followed by the word $B_{2,n} \ldots B_{\lfloor k_{n}/2 \rfloor, n }B_{1,n-1}$. This phenomenon is usually referred as $B_{1,n}$ is an {\em extensible word}.  
	
	Therefore, there exists a word $w$ of length $k_{n}p_{i_{n}}$ such that
	\begin{equation*}\label{eq:extunique}
	\{ i \in \mathbb{Z} ; x|_{[i, i+p_{i_{n}})} =  B_{1,n}  \} \subseteq  \{ i \in \mathbb{Z}; \phi(x)|_{[i-(k_{n} - \lfloor k_{n}/2 \rfloor)p_{i_{n}}, i+\lfloor k_{n}/2 \rfloor p_{i_{n}})} = w\}.
	\end{equation*}
	
	By construction, the set in  the left hand side of the equation is $p_{i_{n+1}}\mathbb{Z}$, so the word $w$ occurs periodically with period $p_{i_{n+1}}$ in $\phi(x)$.  
	
	Recall that $\phi(x) $ is a  concatenation of the words $B_{i,n+1}$  of length $p_{i_{n+1}}$ and, in such decomposition, condition $C2)$ ensures that each word $B_{i,n+1}B_{j,n+1}$, $i,j \in \{2,\ldots,k_{n+1}\}$ appears at least once.
	From the periodicity of  occurrences of $w$ there exists $0 \le p \le p_{i_{n+1}}$ such that for every $i,j\in \{2, \ldots, k_{n+1}\}$,
	\begin{eqnarray}\label{eq:remarquable}
	B_{i,n+1}B_{j,n+1}|_{[p, p +k_{n}p_{i_{n}})} = w.   
	\end{eqnarray}
	
	The following claim shows that this $p$ has to be close to $p_{i_{n+1}}$, that is, $w$ starts with a suffix of $B_{i,n+1}$ and ends with a prefix of $B_{j,n+1}$. 
	
	\begin{claim}\label{lem:combinatoire2}
		Let $p$ be an  integer $0 \le p \le p_{i_{n+1}}$ that  satisfies \eqref{eq:remarquable}. 
		Then, $$\left | p+k_{n}p_{i_{n}}  -p_{i_{n+1}}  - \lfloor k_{n}/2 \rfloor  p_{i_{n}} -1 \right | \le p_{i_{n-1}}(\lfloor k_{n-1}/2 \rfloor +1).$$
	\end{claim}
	\begin{proof}
		Assume the inequality is false. This implies that the word $w =  B_{i,n+1}B_{j,n+1}|_{[p, p+k_{n}p_{i_{n}})}$ has a suffix (or a prefix) with  a factor $B_{\cdot, n-1}$   of $B_{i,n+1}$ (or of $B_{j,n+1}$)  different from the  ordered part  given by condition $C1)$. Since condition $C2)$ ensures that we can find any word $B_{\ell, n-1}$, $\ell \ge 2$,  at any position in $p_{i_{n}} \mathbb{Z} \cap \{p_{i_{n} \lfloor  k_{n} /2\rfloor }, \dots, p_{i_{n+1}}- (k_{n}- \lfloor  k_{n} /2\rfloor)p_{i_{n}}\}$ outside the ordered parts given by condition $C1)$, there exist indices $i',j' \in  \{2, \ldots , k_{n+1}\}$ such that 
		$ B_{i',n+1}B_{j',n+1}|_{[p, p+k_{n}p_{i_{n}})} \neq w$, leading to a contradiction 
	\end{proof}
	
	Call the  word  $B_{i,n+1}B_{j,n+1}|_{[p+ (k_{n}-\rfloor k_{n}/2 \lfloor) p_{i_{n}} -1, p+ (k_{n}-\rfloor k_{n}/2 \lfloor +1) p_{i_{n}}) }$ the {\em  middle word}  of $w$. 
	For any occurrence $\ell$  of $B_{1,n}$ in $x$,  the  unique extension property implies that $\ell$  also is an occurrence of the middle word  of $w$ in $\phi(x)$. Thus, by  Claim \eqref{lem:combinatoire2}, $\ell$ is at distance at most $p_{i_{n-1}}(\lfloor k_{n-1}/2 \rfloor +1)$ from an occurrence of $B_{1,n}$ in $\phi(x)$. 
	Finally, we get $k \in \mathbb{Z}$ with $|k| \le p_{i_{n-1}}(\lfloor k_{n-1}/2 \rfloor +1)$ such that both $x$ and $\sigma^k \phi(x)$ starts with $B_{1,n}$.
	Considering $\sigma^k \phi$ instead of $\phi$,  we can assume that $x$ and $\phi(x)$ start with the word $B_{1,n}$, where $\phi$ admits a block map $\tilde{\phi}$ of radius smaller than 
	$p_{i_{n-1}} + p_{i_{n-1}}(\lfloor k_{n-1}/2 \rfloor +1) \le p_{i_{n}}$.
	
	By the extensible property of $B_{1,n}$ we get that $\phi([B_{1,n}]_0) \subseteq [B_{1,n}]_0$, where $[B_{1,n}]_0$ is the cylinder set starting with word $B_{1,n}$ at zero coordinate. Using that $B_{1,n}$ only appears as a prefix of the words $B_{j,n+1}$ for all $j\in \{1\ldots,k_{n+1}\}$, we get that $\phi(x)$ has to start with some word $B_{j,n+1}$. Since the radius of $\tilde{\phi}$ is lower than $p_{i_{n}}$, the same argument as before shows that there exists another $k\in \mathbb{Z}$ with $|k| \le p_{i_{n}}(\lfloor k_{n}/2 \rfloor +1)< p_{i_{n+1}}$ such that both $x$ and $\sigma^k \phi(x)$ starts with $B_{1,n+1}$. But since $\phi(x)$ starts with some  $B_{j,n+1}$ we have that if $j\neq 1$ then $k\geq p_{i_{n+1}}$. So we get that $j=1$ and $k=0$. Inductively we conclude that $x$ and $\phi(x)$ start with the word $B_{1,n}$ for every $n\geq 1$. Hence, $x$ and $\phi(x)$ are right asymptotic. Since  any word of $x$ has a positive occurrence in $x$, the block map $\tilde{\phi}$ codes the identity map, so $\phi = {\rm Id}$. We conclude that  ${\rm End}(X,\sigma)={\rm Aut}(X,\sigma)=\langle \sigma \rangle$, finishing the proof. 
\end{proof}

In the previous proof we can relax condition $C2)$ by considering all the concatenations of words $B_{i,n-1}$, $i \in \{2, \ldots, k_{n-1}\}$, without imposing restrictions on their number of occurrences, leading to the construction of a coalescent Toeplitz subshift with entropy arbitrarily high but not necessarily uniquely ergodic.  

The fact that we can choose first an infinite odometer and then find a Toeplitz almost one-to-one extension of it with arbitrarily high entropy allow us to deduce  the following.

\begin{theorem}\label{theo:anycountable}
	Any infinite finitely generated abelian group with cyclic torsion subgroup can be realized as the automorphism group of a coalescent Toeplitz subshift with arbitrarily large or zero entropy.
\end{theorem}
\begin{proof}
	Let $G$ be a finitely generated abelian group with cyclic torsion subgroup.
	By the fundamental theorem of finitely generated abelian groups, we may identify $G$ with 
	$\mathbb{Z}^{d}\oplus \mathbb{Z}_{a}$, where $d\geq 1$ and $a=a_1^{s_1}\cdots a_{\ell}^{s_\ell}$ with different prime numbers $a_i$ for $i\in \{1,\ldots,\ell\}$ and $\ell\geq 1$. 
	
	Consider $d$ different primes $r_1,\ldots,r_d$, where each $r_j$ is also chosen to be different from the $a_i$'s for $i\in \{1,\ldots,\ell\}$. 
	Let us consider for each $j\in \{1,\ldots,d\}$  a Toeplitz subshift $(X_j,\sigma)$ such that
	${\rm End}(X_j,\sigma)={\rm Aut}(X_j,\sigma)=\langle \sigma \rangle$, so is isomorphic to $\mathbb{Z}$, and whose associated odometer is $\mathbb{Z}_{(r_j^n)}$, {\em i.e.}, where the scale is given by $(r_j^n)_{n\geq 1}$.   
	Examples of such subshifts are provided by Theorem \ref{Thm:EntropyToeplitz2}, in the positive entropy case and by the examples illustrating  Theorem \ref{Thm:RootsToeplitz} \eqref{item:NonSuperlinear} in the zero entropy case.

	Consider the product space $X=X_1\times \cdots \times  X_d \times \mathbb{Z}_{a}$ with the action $\sigma\colon X\to X$  given by the shift on each $X_j$ for $j\in \{1,\ldots,d\}$ and by the addition by one on the finite system $\mathbb{Z}_{a}$ (recall that in the case a Toeplitz is finite we identify it with the associated odometer). The fact that all the involved odometers (finite and infinite) have different primes in their bases implies that the system $(X,\sigma)$ is a minimal system (we refer to Section 12 in \cite{Downarowicz:2005} for a deeper discussion on disjointness properties of Toeplitz subshifts). Moreover, $(X,\sigma)$ is also a Toeplitz subshift: the $k$-th coordinate of $x=(x_1,\ldots,x_d,m) \in X$ is periodic of period the product of the periods of the $k$-th coordinates of $x_1,\ldots, x_d, m$.  
	By Corollary $\ref{Lemma:AutoProduct}$ we get that $(X,\sigma)$ is coalescent and its automorphism group is abelian and equal to ${\rm Aut}(X_1,\sigma)\oplus\cdots \oplus {\rm Aut}(X_d,\sigma) \oplus {\rm Aut}(\mathbb{Z}_{a},+1)$,
	which is nothing but $\mathbb{Z}^{d} \oplus \mathbb{Z}_{a}$, and thus is isomorphic to $G$.    	
\end{proof}
Finally remark that any finite cyclic group  can be realized as the automorphism  group of a periodic Toeplitz subshift. So to summarize,  any finitely generated abelian group with cyclic torsion subgroup can be realized as the automorphism  group of a Toeplitz subshift.

We finish this section with some open questions: Given a countable subgroup of an odometer $\mathbb{Z}_{(p_n)}$, can we find a Toeplitz subshift whose automorphism group realizes this group?  In particular, does there exist a Toeplitz subshift whose automorphism group contains an infinite countable direct sum of $\mathbb{Z}$?




\bibliographystyle{amsplain}


\begin{dajauthors}
\begin{authorinfo}[Sebastian]
  Sebastian Donoso\\ 
  Centro de Modelamiento Matem\'atico,\\
  CNRS-UMI 2807, Universidad de Chile\\
  Santiago, Chile.\\
  Current address: \\
  Instituto de Ciencias de la Ingenier\'ia \\
  Universidad de O'Higgins\\
  Rancagua, Chile.\\  
  sdonosof\imageat{}gmail\imagedot{}com \\
  \url{https://sites.google.com/site/sebastiandonosofuentes}
\end{authorinfo}
\begin{authorinfo}[Fabien]
  Fabien Durand\\
Laboratoire Ami\'enois de Math\'ematiques Fondamentales et Appliqu\'ees,\\ 
CNRS-UMR 7352, Universit\'{e} de Picardie Jules Verne\\
  Amiens, France\\
  fabien\imagedot{}durand\imageat{}u-picardie\imagedot{}com \\
  \url{http://www.lamfa.u-picardie.fr/fdurand}
\end{authorinfo}
\begin{authorinfo}[Alejandro]
  Alejandro Maass\\
Departamento de Ingenier\'{\i}a
Matem\'atica and Centro de Modelamiento Ma\-te\-m\'a\-ti\-co,\\
CNRS-UMI 2807, Universidad de Chile,  \\
Santiago, Chile.\\
amaass\imageat{}dim\imagedot{}uchile\imagedot{}cl\\
  \url{http://www.dim.uchile.cl/amaass}
\end{authorinfo}
\begin{authorinfo}[Samuel]
  Samuel Petite\\
   Laboratoire Ami\'enois de Math\'ematiques Fondamentales et Appliqu\'ees,\\ 
  CNRS-UMR 7352, Universit\'{e} de Picardie Jules Verne\\
  Amiens, France\\
  samuel\imagedot{}petite\imageat{}u-picardie\imagedot{}com \\
  \url{http://www.lamfa.u-picardie.fr/petite}
\end{authorinfo}
\end{dajauthors}

\end{document}